\documentclass[a4paper,11pt,reqno]{article}

\usepackage{amsmath}
\usepackage{amsthm}
\usepackage{tocloft}
\usepackage{IEEEtrantools}
\usepackage[usenames,dvipsnames]{xcolor}

\usepackage{mathpazo}

\usepackage{hyperref}
\hypersetup{
   colorlinks,
   citecolor=ForestGreen,
   filecolor=blue,
   linkcolor=Maroon,
   urlcolor=orange
}

\newtheorem{theorem}{Theorem}[section]
\newtheorem*{theorem*}{Theorem}
\newtheorem{proposition}[theorem]{Proposition}
\newtheorem*{proposition*}{Proposition}
\newtheorem{lemma}[theorem]{Lemma}
\newtheorem*{lemma*}{Lemma}
\newtheorem{corollary}[theorem]{Corollary}
\theoremstyle{definition}
\newtheorem{definition}[theorem]{Definition}
\newtheorem{remark}[theorem]{Remark}

\numberwithin{equation}{section}

\renewcommand{\theenumi}{\textit{(\roman{enumi})}}

\newcommand{\half}{\frac{1}{2}}
\newcommand{\NN}{\mathbb{N}}
\newcommand{\RR}{\mathbb{R}}
\newcommand{\Support}{\text{supp}}
\newcommand{\V}{V}
\newcommand{\VX}{V(X)}
\newcommand{\Vs}{V(\sigma)}
\newcommand{\VY}{V(Y)}
\newcommand{\Graph}{\mathcal{G}}
\newcommand{\GP}[3]{( #1 \vert #2 )_{#3}}
\newcommand{\DD}[4]{( #1 , #2 \, \vert \, #3 , #4 )}
\newcommand{\dhat}{\hat{d}}        

\newcommand{\DDhat}[4]{\langle #1 , #2 \, \vert \, #3 , #4 \rangle}
\newcommand{\Dhat}{\widehat{D}}    
\newcommand{\ds}{d_{\sigma}}       
\newcommand{\dsi}{d_{\sigma_i}}    
\newcommand{\dt}{d_{\tau}}         
\newcommand{\dD}{d_{\Delta}}       
\newcommand{\dX}{d_X}              
\newcommand{\dG}{d_{\Graph}}       
\newcommand{\tildedG}{\tilde{d}_{\Graph}}   
\newcommand{\dext}{\tilde{d}}
\newcommand{\GPext}[3]{\langle #1 \wr #2 \rangle_{#3}}
\newcommand{\DDext}[4]{\langle #1 , #2 \wr #3 , #4 \rangle}

\newcommand{\sxb}{\ensuremath{
{\raise0.2ex\hbox{$\circ\mspace{-9mu}*$}
\mspace{-1.5mu}\overline{X}}}}
\newcommand{\SymmJoinX}{\sxb}
\newcommand{\SymmJoinMetric}{d_{\ast}}

\newcommand{\length}{\text{length}}

\newcommand{\oneskeleton}{$1$\nobreakdash-\hspace{0pt}skeleton}
\newcommand{\lone}{$l^1$}  
\newcommand{\lonemetric}{\lone\nobreakdash-metric}
\newcommand{\lonepathmetric}{\lone\nobreakdash-path metric}
\DeclareMathOperator{\Isom}{Isom}
\newcommand{\IsomXinvt}{$\Isom(X)$\nobreakdash-\hspace{0pt}invariant}

\title{Extending A Metric On A Simplicial Complex}

\author{Adam Mole\thanks{Supported by SFB 878 - Groups, Geometry and Actions.}}

\begin{document}

\maketitle

\begin{abstract}
 We show how to extend a metric defined on the vertex set of a simplicial complex to the whole simplicial complex, thus correcting a mistake in Mineyev's construction of a flow space.
\end{abstract}


\section{Introduction}\label{Section:Introduction}

In \cite{Mineyev2005FaJiMS}, Mineyev constructed a flow space over  what he called a hyperbolic complex, which is a uniformly locally finite simplicial complex~$X$ whose \oneskeleton~$\Graph$ is a Gromov hyperbolic graph where the metric~$\dG$ is taken to be the length metric induced by giving every edge length~$1$.

This construction involved changing the metric on the vertex set, using a metric~$\dhat$ constructed in \cite{MineyevYu2002TBCCfHG}, and then extending this metric from the vertex set to the whole simplicial complex.

Mineyev used a bilinearity formula to define this extension (see \cite[Equation~31]{Mineyev2005FaJiMS}).
However, this bilinear extension does not define a metric, since it may happen that~$\dext(x,x) > 0$.

In this paper we start more abstractly by showing how to extend a given metric~$\dhat$ that is defined on the vertex set of an arbitrary simplicial complex~$X$ to the whole simplicial complex, under the assumption that~$\dhat$ is linearly bounded by the canonical metric~$\dG$ on the \oneskeleton~$\Graph$ of~$X$.

Then we show that the corresponding extension of Mineyev's metric~$\dhat$ still satisfies the desired properties, such as the double difference extending continuously to the boundary.

\subsection{The Farrell-Jones Conjecture for hyperbolic groups}

As an application of this we look at the Farrell-Jones Conjecture for hyperbolic groups.
The $K$-theoretic Farrell-Jones Conjecture was proven for hyperbolic groups by Bartels-L{\"u}ck-Reich in \cite{BartelsLueckReich2008KFJC}.
The proof involves constructing certain covers of~$G \times \overline{X}$, where~$G$ is a given hyperbolic group and~$\overline{X} = X \cup \partial X$ is the (Gromov) compactification of the Rips complex of~$G$.

The Rips complex is a hyperbolic complex (in the sense of Mineyev in~\cite{Mineyev2005FaJiMS}) and the construction of the covers of~$G \times \overline{X}$ in \cite{BartelsLueckReich2008ECHG} uses the flow space constructed by Mineyev in \cite{Mineyev2005FaJiMS} (where it is called the symmetric join of~$\overline{X}$).
Therefore the mistake with the bilinear extension not defining a metric creates a gap in the proof of the $K$-theoretic Farrell-Jones Conjecture for hyperbolic groups.
In fact, the bilinearity extension formula was explicitly repeated by Bartels-L{\"u}ck-Reich in \cite[Subsection~6.1]{BartelsLueckReich2008ECHG}.

By fixing Mineyev's mistake we will also repair this gap in the proof of the $K$-theoretic Farrell-Jones Conjecture for hyperbolic groups.

\section{Extending a Metric}\label{Section:ExtendingAMetric}

In this section we explain how to extend a metric that is defined on the vertex set of a simplicial complex to the whole simplicial complex.

\subsection{The bilinear extension}\label{Subsection:NaiveApproach}


In any simplicial complex~$X$ with vertex set~$\V := \VX$ every point can be uniquely expressed in barycentric coordinates, namely given~$x \in X$ there are numbers~$x_u \in [0,1]$ for~$u \in \V$ such that $\sum_{u \in \V} x_u = 1$ and $x = \sum_{u \in \V} x_u \, u$.
We call~$(x_u)_{u \in \V}$ the \emph{barycentric coordinates} of~$x$ and we define the \emph{support} of~$x$ to be the set $\Support(x) := \lbrace u \in \V \, \vert \, x_u \neq 0 \rbrace$.

If~$\dhat$ is a metric defined on~$\VX$ then we can extend this to a function on all of~$X$ by taking the bilinear extension of~$\dhat$ using barycentric coordinates.
Thus we define the function~$\Dhat \colon X \times X \rightarrow \RR$ by
\begin{equation}\label{Eqn:Dhat}
 \Dhat(x,y) = \sum_{u,v \in \V} x_u y_v \, \dhat(u,v)\index{$\Dhat$}
\end{equation}
where~$(x_u)_{u \in \V}$ and~$(y_v)_{v \in \V}$ are the barycentric coordinates of the points~$x$ and~$y$ respectively.

This function~$\Dhat$ fails to be a metric because there are points~$x \in X$ for which~$\Dhat(x,x) > 0$.
In fact, for any point~$x$ that is not a vertex we can find two distinct vertices~$u_1$ and~$u_2$ such that both~$x_{u_1} \neq 0$ and~$x_{u_2} \neq 0$.
Then
\begin{align*}
 \Dhat(x,x)
   &= \sum_{u,u' \in \V} x_u x_{u'} \, \dhat(u,u')
   \\
   &\geq x_{u_1} x_{u_2} \, \dhat(u_1,u_2)
   \\
   &> 0.
\end{align*}
So this na\"{i}ve approach does not give a metric.
However, the function~$\Dhat$ is symmetric, non-negative, and it satisfies the triangle inequality;
\begin{lemma}\label{Lemma:BilinearExtTriangleInequality}
 Let~$X$ be a simplicial complex and let~$\dhat$ be a metric on the vertex set of~$X$.
 Let~$\Dhat$ be the bilinear extension of~$\dhat$ using barycentric coordinates, as in equation~\eqref{Eqn:Dhat}.
 Then for all~$x,y,z \in X$,
 \[
  \Dhat(x,z) \leq \Dhat(x,y) + \Dhat(y,z).
 \]
\end{lemma}
\begin{proof}
 Write~$x,y,z$ in barycentric coordinates, and recall that the sum of the barycentric coordinates associated to a point is always~$1$.
 Then we use the triangle inequality of~$\dhat$ and calculate;
 \begin{align*}
  \Dhat(x,z)
    &= \sum_{u,w \in \V} x_u z_w \, \dhat(u,w)
      \\
    &= \sum_{u,v,w \in \V} x_u y_v z_w \, \dhat(u,w)
      \\
    &\leq \sum_{u,w \in \V} x_u y_v z_w 
                \big( \dhat(u,v) + \dhat(v,w) \big)
      \\
    &= \sum_{u,v,w \in \V} x_u y_v z_w \, \dhat(u,v)
       + \sum_{u,v,w \in \V} x_u y_v z_w \, \dhat(v,w)
      \\
    &= \sum_{u,v \in \V} x_u y_v \, \dhat(u,v)
       + \sum_{v,w \in \V} y_v z_w \, \dhat(v,w)
      \\
    &= \Dhat(x,y) + \Dhat(y,z).
    \qedhere
 \end{align*}
\end{proof}

Therefore the function~$\Dhat$ only fails to be a metric locally.
We try to correct this by introducing a new metric~$\dX$ on~$X$ and defining the extension of~$\dhat$ to be the minimum of~$\Dhat$ and~$\dX$.

We have to be careful how we construct the metric~$\dX$ so that it is only smaller than~$\Dhat$ locally, and in particular it is smaller than~$\dhat$ on vertices.

\subsection{The $l^1$-path metric}\label{Subsection:lonePathMetric}

We want to construct a metric~$\dX$ on a simplicial complex~$X$ to help fix the problem encountered with the bilinear extension in Subsection~\ref{Subsection:NaiveApproach}.

There is a canonical metric on a simplex~$\sigma$ given by the \lonemetric\ applied to the barycentric coordinates;
for all points~$x,y \in \sigma$ set
\begin{equation}\label{Eqn:l1Metric}
 \ds(x,y) := \half \sum_{u \in \Vs} \lvert x_u - y_u \rvert
\end{equation}
where the factor~$\half$ is to ensure that the distance between vertices is~$1$.
Note that if~$\tau$ is a subsimplex of~$\sigma$ then~$\dt \equiv \ds \vert_{\tau}$ so given two points in a common simplex it does not matter which common simplex we use for the \lonemetric.

If~$X$ is a simplicial complex we can piece together these metrics by defining paths and taking a length metric as follows.

\begin{definition}\label{Defn:l1Metric}
 Let~$X$ be a simplicial complex.
 For any two points~$x,y \in X$ define a \emph{path in~$X$ from~$x$ to~$y$} to be a sequence of points~${x=a_0,a_1, \ldots, a_r=y}$ in~$X$ such that for every~$i = 1,\ldots,r$ there is a simplex~$\sigma_i$ of~$X$ that contains both~$a_{i-1}$ and~$a_i$.
 The \emph{length} of such a path is the sum~$\sum_{i=1}^r \dsi(a_{i-1},a_i)$.
 
 The length of a path does not depend on how we choose the simplices~$\sigma_i$ because the \lonemetric\ of a simplex restricts to the \lonemetric\ on subsimplices.
 
 Then we define a metric on~$X$ by setting~$\dX(x,y)$\index{$\dX$} to be the infimum of lengths of all such paths from~$x$ to~$y$.
 We call this metric the \emph{\lonepathmetric\ on~$X$}.
\end{definition}

\begin{lemma}\label{Lemma:l1PathMetricRestrictions}
 Let~$X$ be a simplicial complex.
 \begin{enumerate}
  \item\label{Item:dloneRestriction} For any simplex~$\sigma$ in~$X$, the restriction of~$\dX$ to~$\sigma$ coincides with~$\ds$.
  \item\label{Item:dloneDiameterOfSimplex} The diameter of any simplex is~$1$, unless the simplex is a vertex.
  \item\label{Item:dloneDisjointSupports} If~$x,y \in X$ have disjoint supports then~$\dX(x,y) \geq 1$.
 \end{enumerate}
\end{lemma}
\begin{proof}
 \ref{Item:dloneRestriction}
 Suppose~$x,y \in X$ lie in a common simplex~$\sigma$.
 By considering the trivial path~$x=a_0,a_1=y$ we get~$\dX(x,y) \leq \ds(x,y)$ so we only need to show the inequality in the other direction.
 We need to show that the length of any path in~$X$ from~$x$ to~$y$ is at least~$\ds(x,y)$.

 Consider a path~$a_0,a_1,\ldots,a_r$ from~$x$ to~$y$ in~$X$.
 Since it is a path we can find simplices~$\sigma_i$ for~$i=1,\ldots,r$ such that~$\sigma_i$ contains both~$a_{i-1}$ and~$a_i$.
 Let~$Y$ be the subcomplex of~$X$ consisting of~$\sigma$ and all the~$\sigma_i$.
 Denote the vertex set of~$Y$ by~$\VY$.
 Let~$\Delta=\Delta^{\lvert\VY\rvert-1}$ be the standard simplex of dimension~$\lvert\VY\rvert - 1$ and identify the vertex set of~$\Delta$ with~$\VY$.

 There is a canonical inclusion~$Y \hookrightarrow \Delta$ and so we can think of~$Y$ as a subcomplex of~$\Delta$.
 Then for each~$i$ we have~$\dsi(a_{i-1},a_i) = \dD(a_{i-1},a_i)$ and the triangle inequality of~$\dD$ tells us that the length of the path~$a_0,\ldots,a_r$ is at least~$\dD(a_0,a_r) = \ds(x,y)$.
 
 We can do this for any path from~$x$ to~$y$ so~$\dX(x,y) \geq \ds(x,y)$.
 Therefore we can conclude that~$\dX$ restricts to~$\ds$ on the simplex~$\sigma$.
 
 \ref{Item:dloneDiameterOfSimplex}
 It follows from part~\ref{Item:dloneRestriction} and equation~\eqref{Eqn:l1Metric} that if two points~$x,y$ lie in a common simplex~$\sigma$ of~$X$ then~$\dX(x,y) = \ds(x,y) \leq 1$.
 Taking~$x,y$ to be distinct vertices of~$\sigma$ shows that the bound is strict.
 
 \ref{Item:dloneDisjointSupports}
 If~$u$ is a vertex in the support of~$x$ but not in the support of~$y$ then any path from~$x$ to~$y$ must eventually take the $u$-coordinate down from~$x_u$ to zero.
 Formally, if~$a_0,\ldots,a_r$ is a path from~$x$ to~$y$ let~$a_{i,u}$ denote the $u$\nobreakdash-\hspace{0pt}coordinate of~$a_i$.
 Then
 \begin{align*}
  \sum_{i=1}^r \half \left\lvert a_{i-1,u} - a_{i,u} \right\rvert
  &\geq
  \half \left\lvert \sum_{i=1}^r \left( a_{i-1,u} - a_{i,u} \right) \right\rvert
  \\
  &=
  \half \left\lvert a_{0,u} - a_{r,u} \right\vert
  = \half x_u
 \end{align*}
 and thus the $u$-coordinate contributes at least~$\half x_u$ to~$\dX(x,y)$.
 
 Similarly if~$v$ is in the support of~$y$ but not the support of~$x$ then any path must take the $v$-coordinate from zero up to~$y_v$, hence the $v$-coordinate contributes at least~$\half y_v$ to~$\dX(x,y)$.
 
 Therefore if the support of~$x$ is disjoint from the support of~$y$ then we get at least~$\half$ from the coordinates in the support of~$x$ and at least~$\half$ from the coordinates in the support of~$y$.
 Thus~$\dX(x,y) \geq 1$.
\end{proof}

We want to use the metric~$\dX$ to define an extension of a metric~$\dhat$ that is given on the vertex set of~$X$.
Hence we need to know how the metric~$\dX$ behaves on~$\VX$.

Before that we need some new notation:
Let~$\Graph$ be a graph (for our purposes all graphs are simple, i.e. every edge has two distinct end-points and there is at most one edge between any pair of vertices).
The \emph{word metric} on~$\Graph$ is the length metric~$\dG$\index{$\dG$} obtained by giving every edge length~$1$.

Note that a graph can be considered as a $1$-dimensional simplicial complex and so has an \lonepathmetric.
However this metric coincides with the word metric on~$\Graph$ so the notation~$\dG$ is unambiguous.

Now we prove that the \lonepathmetric\ on the vertices of a simplicial complex coincides with the word metric of the $1$-skeleton, i.e. that adding higher dimensional simplices to a graph does not change the \lonepathmetric\ on vertices.

\begin{lemma}\label{Lemma:distanceloneRestrictionToV}
 Let~$X$ be a simplicial complex, and let~$\Graph$ be the $1$-skeleton of~$X$.
 For any~$u,v \in \V := \VX$ we have
 \[
  \dX(u,v) = \dG(u,v).
 \]
\end{lemma}
\begin{proof}
 If $u = v$ then $\dX(u,v) = 0 = \dG(u,v)$ so we may assume that~$u \neq v$.

 Any path in the $1$-skeleton is also a path in~$X$, so~$\dG(u,v) \geq \dX(u,v)$.
It remains to prove that the length of any path in~$X$ from~$u$ to~$v$ is at least~$\dG(u,v)$.

 For~$k = 0, \ldots, d := \dG(u,v)$ let~$S_k = \lbrace z \in \V \, \vert \, \dG(u,z) = k \rbrace$.
 This is the sphere of radius~$k$ and centre~$u$ in~$(\Graph,\dG)$.
 The idea is that the support of any path from~$u$ to~$v$ must meet every~$S_k$ and moreover the weight must come from~$u$ and through each~$S_k$ before arriving at~$v$ (the path could be longer but we are only looking for a lower bound).
 
 Let~$u = a_0 , a_1, \ldots , a_r = v$ be a path in~$X$ from~$u$ to~$v$.
 So for any~$i$ there is always a simplex~$\sigma_i$ that contains both~$a_{i-1}$ and~$a_i$.
 Write each~$a_i = \sum_{z \in \V} a_i^z z$ in barycentric coordinates and for all~$k$ set
 \[
  w_k(a_i) = \sum_{z \in S_k} a_i^z \in [0,1]
 \]
 which is the \emph{weight} of the point~$a_i$ in the sphere~$S_k$.
 
 The length of this path~$\underline{a}$ is
 \begin{align*}
  \length\big(\underline{a}\big)
  &=
  \sum_{i=1}^r
  \dsi(a_{i-1},a_i)
       \\
  &=
  \sum_{i=1}^r
  \sum_{z \in \VX}
  \half
  \left\lvert a_{i-1}^z - a_i^z \right\rvert
       \\
  &\geq
  \half
  \sum_{i=1}^r
  \sum_{k = 0}^{d}
  \sum_{z \in S_k}
  \left\lvert a_{i-1}^z - a_i^z \right\rvert
       \\
  &\geq
  \half
  \sum_{i=1}^r
  \sum_{k = 0}^{d}
  \left\lvert \sum_{z \in S_k} a_{i-1}^z - \sum_{z \in S_k} a_i^z \right\rvert
       \\
  &=
  \half
  \sum_{k = 0}^{d}
  \sum_{i=1}^r
  \left\lvert w_k (a_{i-1}) - w_k (a_i) \right\rvert.
 \end{align*}
 Now we try to find a lower bound for the sums~$\sum_{i=1}^r \left\lvert w_k (a_{i-1}) - w_k (a_i) \right\rvert$.

 If~$k=0$ then
 \begin{align*}
  \sum_{i=1}^r
  \left\lvert w_0 (a_{i-1}) - w_0 (a_i) \right\rvert
  &\geq
  \left\lvert
   \sum_{i=1}^r \big( w_0 (a_{i-1}) - w_0 (a_i) \big)
  \right\rvert
       \\
  &=
  \left\lvert
   w_0 ( a_0 ) - w_0 (a_r)
  \right\rvert
       \\
  &=
  1
 \end{align*}
 since~$a_0 = u \in S_0$ and~$a_r = v \in S_d$.
 Similarly for~$k = d$ we get
 \begin{align*}
  \sum_{i=1}^r
  \left\lvert w_d (a_{i-1}) - w_d (a_i) \right\rvert
  \geq
  \left\lvert w_d (a_0) - w_d (a_r) \right\rvert
  =
  1.
 \end{align*}

 So we are left with~$1 \leq k < d$.
 We claim that for such a~$k$ we can find some~$i_k \in \lbrace 1,\ldots,d-1 \rbrace$ such that~$w_k(a_{i_k}) = 1$.
 
 If for now we assume the claim is true, then
 \begin{IEEEeqnarray*}{rCl}
  \IEEEeqnarraymulticol{3}{l}{
   \sum_{i=1}^r
   \left\lvert
    w_k (a_{i-1}) - w_k (a_i)
   \right\rvert
  }
        \\
  \qquad \quad &=&
  \sum_{i=1}^{i_k}
  \left\lvert
   w_k (a_{i-1}) - w_k (a_i)
  \right\rvert
  + \sum_{i=i_k+1}^r
  \left\lvert
   w_k (a_{i-1}) - w_k (a_i)
  \right\rvert
        \\
  &\geq&
  \left\lvert
   \sum_{i=1}^{i_k}
   \big( w_k (a_{i-1}) - w_k (a_i) \big)
  \right\rvert
  + \left\lvert
     \sum_{i=i_k+1}^{i_k} \big( w_k (a_{i-1}) - w_k (a_i) \big)
  \right\rvert
        \\
  &=&
  \left\lvert
   w_k (a_0) - w_k (a_{i_k})
  \right\rvert
  +
  \left\lvert
   w_k (a_{i_k}) - w_k (a_r)
  \right\rvert
        \\
  &=& 2
 \end{IEEEeqnarray*}
 since $w_k(a_0) = 0 = w_k(a_r)$ for $1 \leq k < d$, and we would get
 \begin{align*}
  \length\big(\underline{a}\big)
  &\geq
  \half
  \sum_{k = 0}^{d}
  \sum_{i=1}^r
  \left\lvert w_k (a_{i-1}) - w_k (a_i) \right\rvert
       \\  
  &\geq
  \half
  \big(
   1 + 2(d-2) + 1
  \big)
       \\
  &= d = \dG(u,v).
 \end{align*}
 
 Therefore, to finish the proof of the lemma we need to prove the claim that for all~$k=1,\ldots,d-1$ there is some~$i_k$ with~$w_k (a_{i_k}) = 1$.
 
 For all~$i$ there is a simplex containing both~$a_i$ and~$a_{i-1}$ hence every vertex in the support of~$a_i$ is joined to any vertex in the support of~$a_{i-1}$ by an edge.
 In particular, if $w_{k-1}(a_{i-1}) \neq 0$ then $w_{k+1}(a_i) = 0$ and similarly if $w_{k+1}(a_i) \neq 0$ then $w_{k-1}(a_{i-1}) = 0$ because we may move at most one further away from~$u$ with successive~$a_i$'s.
 So if we pick~$i_k$ minimal with $w_{k+1}(a_{i_k+1}) \neq 0$ then $w_{k-1}(a_{i_k}) = 0$ and $w_{k+1}(a_{i_k}) = 0$.
 So all the weight of~$a_{i_k}$ has to be at the sphere~$S_k$, but $\sum_k w_k (a_{i_k}) = 1$.
 Therefore $w_k (a_{i_k}) = 1$.
 
 This proves the claim, and with it also finishes the proof of the lemma by the earlier argument.
\end{proof}

This shows us how the \lonepathmetric\ behaves on the vertex set of a simplicial complex.
We are interested in extending a metric~$\dhat$ which is defined on the vertex set of a simplicial complex to the whole simplicial complex and we hope to do this by taking the minimum of the bilinear extension~$\Dhat$ and the \lonepathmetric~$\dX$.

It follows from Lemma~\ref{Lemma:distanceloneRestrictionToV} that we need~$\dhat(u,v) \leq \dG(u,v)$ for all vertices~$u,v \in \VX$ in order for the minimum to be~$\dhat$ on vertices.
This will not hold in general and so we get a condition on the metric~$\dhat$, namely we ask that there is a constant~$C > 0$ such that for all vertices~$u,v \in \VX$ we have~$\dhat(u,v) \leq C \, \dG(u,v)$.

Then instead of taking the minimum of~$\Dhat$ and~$\dX$ we re-scale the \lonepathmetric\ by~$C$ and take the minimum of~$\Dhat$ and~$C \, \dX$.

\subsection{Extension of a metric}\label{Subsection:Extension}

Suppose we are given a metric~$\dhat$ that is defined on the vertex set of a simplicial complex~$X$.
We wish to extend this to a metric on all of~$X$, and the idea is to take a minimum of the bilinear extension~$\Dhat$ of~$\dhat$ and a re-scaled \lonepathmetric.
Although the triangle inequality holds for both~$\Dhat$ (by Lemma~\ref{Lemma:BilinearExtTriangleInequality}) and the \lonepathmetric, for the triangle inequality to also hold on the minimum we need to look at what happens in the mixed case when the minimum is~$\Dhat$ on one pair and the \lonepathmetric\ on the other.

\begin{lemma}\label{Lemma:DistMinAlmostTriangleInequality}
 Let~$X$ be a simplicial complex and let~$\dhat$ be a metric on the vertex set~$\V = \VX$.
 Suppose there exists a constant~$C>0$ such that for all vertices $u,v \in \V$ we have~$\dhat(u,v) \leq C \, \dG(u,v)$, where~$\dG$ is the word metric on the $1$-skeleton~$\Graph$ of~$X$.
 Then for all~$x,y,z \in X$,
 \begin{equation}\label{Eqn:MinTriangleInequalityMixedCase}
  \Dhat(x,z) \leq \Dhat(x,y) + 2C \, \dX(y,z).
 \end{equation}
\end{lemma}
\begin{proof}
 First we will consider the special case where~$y$ and~$z$ lie in a common simplex~$\sigma$.
 By definition, we can expand bilinearly and obtain
 \begin{align*}
  \Dhat(x,z)
     &= \sum_{u,w \in \V} x_u z_w \, \dhat(u,w)
     \\
     &= \sum_{u,w \in \V} x_u y_w \, \dhat(u,w)
              + \sum_{u,w \in \V} x_u (z_w - y_w) \, \dhat(u,w)
     \\
     &= \Dhat(x,y) + \sum_{u,w \in \V} x_u (z_w - y_w) \, \dhat(u,w).
 \end{align*}
 The first term is as in inequality~\eqref{Eqn:MinTriangleInequalityMixedCase} but we need to bound that second term by~$2 C \, \dX(y,z)$.
 
 Fix a vertex~$v_0$ of~$\sigma$.
 Then~$\dhat(u,w) \leq \dhat(u,v_0) + \dhat(v_0,w)$ for any vertices~$u,w$ of~$X$.
 The distance~$\dhat(u,v_0)$ is independent of the vertex~$w$ so we have~$\sum_{u,w \in \V} x_u (z_w - y_w) \, \dhat(u,v_0) = 0$.
 Moreover, since~$v_0$ is joined to any vertex of~$\sigma$ by an edge, we know~$\dG(v_0,w) \leq 1$ for~$w \in \Support(\sigma)$ and so~$\dhat(v_0,w) \leq C$.
 Therefore
 \begin{align*}
  \sum_{u,w \in \V} x_u (z_w - y_w) \, \dhat(u,w)
     &\leq \sum_{u,w \in \V} x_u |z_w - y_w| C
     \\
     &= 2C \, \ds(y,z).
 \end{align*}
 But then from Lemma~\ref{Lemma:l1PathMetricRestrictions}\ref{Item:dloneRestriction} we get~$\ds(y,z) = \dX(y,z)$.

 Putting all this together gives the desired inequality in this special case.
 
 For the general case, consider a path~$y=a_0 , a_1, \ldots,a_r=z$ from~$y$ to~$z$ (in the sense of Definition~\ref{Defn:l1Metric}).
 Then
 \[
  \Dhat(x,z)
     =
  \Dhat(x,y)
  + \sum_{i=1}^r \left( \Dhat(x,a_i) - \Dhat(x,a_{i-1}) \right)
 \]
 We know that~$a_{i-1}$ and~$a_i$ lie in a common simplex so by the special case we have~$\Dhat(x,a_i) - \Dhat(x,a_{i-1}) \leq 2C \, \dX(a_i,a_{i-1})$.
 Thus
 \begin{align*}
  \Dhat(x,z)
       &\leq 
  \Dhat(x,y) + \sum_{i=1}^r 2C \, \dX(a_i,a_{i-1})
       \\
       &=
  \Dhat(x,y) + 2C \, \length(\underline{a})
 \end{align*}
 where~$\length(\underline{a})$ is the length of the path~$a_0,\ldots,a_r$.
 
 This holds for any path from~$y$ to~$z$ so we conclude
 \[
  \Dhat(x,z) \leq \Dhat(x,y) + 2C \, \dX(y,z).
  \qedhere
 \]
\end{proof}

So the minimum of~$\Dhat$ and~$2C \, \dX$ will define a metric on~$X$ (which will be formally proven later in Theorem~\ref{Thm:Extension}) but we only wanted to change the bilinear extension~$\Dhat$ locally.
To achieve this, we increase the scaling factor from~$2C$ to~$3C$ and use the following lemma.
\begin{lemma}\label{Lemma:dhatSmaller}
 Let~$X$ be a simplicial complex, and let~$\Graph$ denote its $1$-skeleton.
 Let~$\dhat$ be a metric on the vertex set of~$X$ with a constant~$C>0$ such that for all vertices~$u,v \in \V := \VX$ we have~$\dhat(u,v) \leq C \, \dG(u,v)$.
 For all~$x,y \in X$ if the support of~$x$ is disjoint from the support of~$y$ then
 \[
  \Dhat(x,y) \leq 3C \, \dX(x,y).
 \]
\end{lemma}
\begin{proof}
 Start by considering vertices~$u,v \in \VX$.
 Using Lemma~\ref{Lemma:distanceloneRestrictionToV} it follows immediately from the assumptions that
 \[
  \Dhat(u,v) = \dhat(u,v) \leq C \, \dG(u,v) = C \, \dX(u,v).
 \]
 
 Now consider the general case of points~$x,y \in X$ with disjoint supports.
 Since the support of~$x$ is disjoint from the support of~$y$, we must have $\dX(x,y) \geq 1$ by Lemma~\ref{Lemma:l1PathMetricRestrictions}\ref{Item:dloneDisjointSupports}.
 Moreover, for any~$u \in \Support(x)$ Lemma~\ref{Lemma:l1PathMetricRestrictions}\ref{Item:dloneDiameterOfSimplex} tells us that $\dX(u,x) \leq 1$.
 In particular, for any~$u \in \Support(x)$ we have $\dX(u,x) \leq \dX(x,y)$.
 Similarly for any vertex~$v \in \Support(y)$ we get $\dX(y,v) \leq \dX(x,y)$.
 Therefore
 \begin{align*}
  \Dhat(x,y)
  &=
  \sum_{u,v \in \V} x_u y_v \, \dhat(u,v)
  \\
  &\leq
  \sum_{u,v \in \V} x_u y_v C \, \dX(u,v)
  \\
  &\leq
  \sum_{u,v \in \V} x_u y_v C
   \big( \dX(u,x) 
     + \dX(x,y)
     + \dX(y,v)
   \big)
  \\
  &\leq
  \sum_{u,v \in \V} x_u y_v \, 3C \, \dX(x,y)
  \\
  &=
  3C \, \dX (x,y).
  \qedhere
 \end{align*}
\end{proof}

Now we can define our extension.
\begin{definition}\label{Defn:ExtensionMetric}
 Let~$X$ be a simplicial complex and let~$\dhat$ be a metric on the vertex set of~$X$ with a constant~$C>0$ such that for all vertices~$u,v \in \VX$ we have~$\dhat(u,v) \leq C \, \dG(u,v)$ where~$\Graph$ is the $1$-skeleton of~$X$.
 
 Define the function~$\dext \colon X \times X \rightarrow \RR$ by
 \begin{equation}\label{Eqn:MetricExtension}
 \dext(x,y) := \min \left\lbrace  \Dhat(x,y) \, , \, 3C \, \dX(x,y)  \right\rbrace
 \end{equation}
 where~$\Dhat$ is the bilinear extension of~$\dhat$ using barycentric coordinates and~$\dX$ is the \lonepathmetric\ on~$X$.
\end{definition}

\begin{theorem}\label{Thm:Extension}
 Under the same conditions as in Defintion~\ref{Defn:ExtensionMetric} the function~$\dext$ is a metric on~$X$.
 Moreover, for all~$x,y \in X$ if the support of~$x$ is disjoint from the support of~$y$ then~$\dext(x,y) = \Dhat(x,y)$.
 In particular~$\dext$ is an extension of~$\dhat$.
\end{theorem}
\begin{proof}
 The function~$\dext$ is symmetric and non-negative since it is the minimum of two functions that are both symmetric and non-negative.
 
 For any~$x \in X$ we have
 \[
  \dext(x,x) \leq 3C \, \dX(x,x) = 0.
 \]
 If~$x,y \in X$ are distinct points then~$\dX(x,y) > 0$ since~$\dX$ is a metric on~$X$, and for any vertices~$u \in \Support(x)$ and~$v \in \Support(y)$ we get
 \[
  \Dhat(x,y) \geq x_u y_v \, \dhat(u,v) >0.
 \]
 Thus~$\dext(x,y) >0$, and therefore~$\dext$ is positive definite.
 
 It remains to prove that~$\dext$ satisfies the triangle inequality.
 Given three points~$x,y,z \in X$, if $\dext(x,y) = 3C \, \dX(x,y)$ and $\dext(y,z) = 3C \, \dX(y,z)$ then we can use the triangle inequality of~$\dX$.
 Similarly if~$\dext(x,y) = \Dhat(x,y)$ and~$\dext(y,z) = \Dhat(y,z)$ then
 \[
  \dext(x,z)
  \leq \Dhat(x,z)
  \leq \Dhat(x,y) + \Dhat(y,z)
  = \dext(x,y) + \dext(y,z)
 \]
 using Lemma~\ref{Lemma:BilinearExtTriangleInequality}.
 For the mixed case suppose (without loss of generality) that $\dext(x,y) = \Dhat(x,y)$ and $\dext(y,z) = 3C \,\dX(y,z)$.
 Then we can use Lemma~\ref{Lemma:DistMinAlmostTriangleInequality} to get
 \[
  \dext(x,z)
  \leq \Dhat(x,z)
  \leq \Dhat(x,y) + 3C \, \dX(y,z)
  = \dext(x,y) + \dext(y,z).
 \]
 Thus the function~$\dext$ satisfies the triangle inequality, and hence we have shown that~$\dext$ is indeed a metric.
 
 Moreover, if~$x,y \in X$ have disjoint supports then~$\Dhat(x,y) \leq 3C \, \dX(x,y)$ by Lemma~\ref{Lemma:dhatSmaller}, and so~$\dext(x,y) = \Dhat(x,y)$.
\end{proof}

The condition on the metric~$\dhat$ that such a~$C$ exists is satisfied when~$\dhat$ is $(A,B)$-quasi-isometric to~$\dG$ by setting~$C= A+B$ and remembering that on~$\VX$ the metric~$\dG$ takes values in~$\NN$.
Thus we get the following corollary of Theorem~\ref{Thm:Extension}.

\begin{corollary}\label{Cor:QIExtension}
 Let~$X$ be a simplicial complex and~$\Graph$ be its $1$-skeleton.
 If~$\dhat$ is a metric defined on the vertex set of~$X$ that is quasi-isometric to the word-metric~$\dG$ then there is an extension~$\dext$ of~$\dhat$ to all of~$X$ that has the following properties;
 \begin{enumerate}
  \item\label{Item:QIExtDisjointSupports} If~$x,y \in X$ have disjoint supports then~$\dext(x,y) = \Dhat(x,y)$.
  \item\label{Item:QIExtQI} There is a constant~$B' \geq 0$ such that for all~$x,y \in X$
  \[
   \Dhat(x,y) - B' \leq \dext(x,y) \leq \Dhat(x,y) + B'.
  \]
 \end{enumerate}
\end{corollary}
\begin{proof}
 Suppose~$\dhat$ is $(A,B)$-quasi-isometric to~$\dG$.
 Set~$C = A+B$
 and apply Theorem~\ref{Thm:Extension} to get an extension~$\dext$ of~$\dhat$.
 The theorem immediately gives us property~\ref{Item:QIExtDisjointSupports}.
 
 We still need to show that such a $B'$ exists.
 
 Fix~$x,y \in X$.
 If the support of~$x$ is disjoint from the support of~$y$ then we already know that $\dext(x,y) = \Dhat(x,y)$, so suppose there exists some $v \in \Support(x) \cap \Support(y)$.
 Then for any $u \in \Support(x) \backslash \lbrace v \rbrace$ there is an edge joining~$u$ and~$v$, thus $\dG(u,v) = 1$ and then using $\dhat(u,v) \leq (A+B) \dG(u,v)$ gives us
 \[
  \Dhat(x,v)
  =
  \sum_{u \in \VX} x_u \, \dhat(u,v)
  \leq
  \sum_{u \in \VX} x_u (A+B)
  =
  (A+B)
 \]
 and similarly~$\Dhat(v,y) \leq (A+B)$.
 Therefore
 \begin{align*}
  \Dhat(x,y)
  &\leq
  \Dhat(x,v) + \Dhat(v,y)
  \leq
  2(A+B)
  \leq
  \dext(x,y) + 2(A+B).
 \end{align*}
 So
 \[
  \Dhat(x,y) - 2(A+B)
  \leq
  \dext(x,y)
  \leq
  \Dhat(x,y).
 \]
 Setting~$B' = 2(A+B)$ completes the proof.
\end{proof}

\section{Applications}\label{Section:Applications}

In this section we look at applications of this extension by using it to correct a flaw in~\cite{Mineyev2005FaJiMS}, which was later copied in~\cite{BartelsLueckReich2008ECHG}.

\subsection{Mineyev's flow space}\label{Subsection:MineyevsFlowSpace}

In~\cite{Mineyev2005FaJiMS}, Mineyev constructed a flow space associated to what he called a hyperbolic complex.
The problem first arises in his definition of a hyperbolic complex in~\cite[Subsection~5.3]{Mineyev2005FaJiMS}.
\begin{definition}\label{Defn:HyperbolicComplex}
 A \emph{hyperbolic complex} is a connected, uniformly locally finite, simplicial complex~$X$ whose~$1$-skeleton~$\Graph$ is a hyperbolic graph with respect to its word metric~$\dG$.
 Then Mineyev uses the bilinear extension formula
 \[
  d(x,y) = \sum_{u,v \in \VX} x_u y_v \, \dG(u,v)
 \]
 to define a metric on~$X$.
 However, as shown at the start of Subsection~\ref{Subsection:NaiveApproach} this function is not a metric unless~$X$ is just a collection of vertices.
 
 Therefore we alter the definition of a hyperbolic complex to say~$X$ has the metric~$\tildedG$ obtained from Theorem~\ref{Thm:Extension} by setting~$\dhat$ to be the restriction of~$\dG$ to~$\VX$ (and~$C=1$).
\end{definition}

Mineyev later uses the bilinear extension formula in \cite[Subsection~6.2]{Mineyev2005FaJiMS} but this time the metric~$\dhat$ is taken to be a metric constructed by Mineyev-Yu in an earlier paper, namely~\cite{MineyevYu2002TBCCfHG}.

This metric~$\dhat$ is strongly bolic (see \cite[Definition 15]{MineyevYu2002TBCCfHG}), invariant under simplicial automorphisms of~$X$, quasi-isometric to the word metric~$\dG$, and has the property that there exist constants~$C \geq 0$ and~$\mu \in [0,1)$ such that for all~$u,u',v,v' \in \VX$ if both~$\dG(u,u') \leq 1$ and~$\dG(v,v') \leq 1$ then
\begin{equation}\label{Eqn:DDhatConvergence}
 \left\lvert
  \dhat(u,v)
  - \dhat(u',v)
  - \dhat(u,v')
  + \dhat(u'v')
 \right\rvert
 \leq
 C \, \mu^{\dG(a,b)}.
\end{equation}
This metric was originally constructed for a hyperbolic group, see \cite[Proposition~13 and Theorem~17]{MineyevYu2002TBCCfHG}.
This was recalled as \cite[Theorem~26]{Mineyev2005FaJiMS} and after this theorem Mineyev explained what modifications are necessary for the construction to work for an arbitrary hyperbolic complex.

Let~$\Dhat$ be the bilinear extension of~$\dhat$ (as defined in equation~\eqref{Eqn:Dhat}).
The expression
\begin{equation}\label{Eqn:DoubleDifferenceHat}
 \DDhat{x}{x'}{y}{y'} :=
 \Dhat(x,y) - \Dhat(x',y) - \Dhat(x,y') + \dhat(x',y')
\end{equation}
is called the \emph{double difference} of the four points $x,x',y,y' \in X$ and is very important in Mineyev's work.
The crucial property of this double difference is that it extends continuously to all of~$\overline{X} = X \cup \partial X$, where~$\partial X$ is the (Gromov) boundary of~$X$ (see \cite[Section~III.H.3]{BridsonHaefliger1999}).
This was stated as \cite[Theorem~35]{Mineyev2005FaJiMS}.

We use the notation
\begin{equation}\label{Eqn:DoubleDifferenceExt}
 \DDext{x}{x'}{y}{y'} :=
 \dext(x,y) - \dext(x',y) - \dext(x,y') + \dext(x',y')
\end{equation}
to denote the double difference with respect to the extension~$\dext$ of~$\dhat$ given by Theorem~\ref{Thm:Extension}.
We need to show that the double difference with respect to~$\dext$ also extends continuously to~$\overline{X}$.

\begin{remark}\label{Rk:DoubleDifferences+Equivalent}
 We can relate the double difference with respect to~$\dext$ to the double difference with respect to~$\dhat$ as follows:
 Since~$\dhat$ is quasi-isometric to~$\dG$ Corollary~\ref{Cor:QIExtension}\ref{Item:QIExtQI} says that there is some constant~$B' \geq 0$ such that~$\dext$ for all~$x,y \in X$ we have $\lvert \dext(x,y) - \Dhat(x,y) \rvert \leq B'$.
 Hence for any~$x,x',y,y' \in X$;
 \begin{equation}\label{Eqn:DoubleDifferences+Equivalent}
  \DDhat{x}{x'}{y}{y'} - 4B' \leq \DDext{x}{x'}{y}{y'} \leq \DDhat{x}{x'}{y}{y'} + 4B'.
 \end{equation}
\end{remark}

\begin{theorem}\label{Thm:DDextExtensionToBoundary}
 Let~$X$ be a hyperbolic complex.
 Set~$\overline{S}$ to be the subset of~$\overline{X}^4 := \overline{X} \times \overline{X} \times \overline{X} \times \overline{X}$ consisting of points~$(x,x',y,y')$ such that no point of~$\partial X$ appears more than twice in the coordinates~$x,x',y,y'$. 
 Then the double difference with respect to~$\dext$ extends to a continuous \IsomXinvt\ function~$\overline{S} \rightarrow \overline{\RR} = [-\infty,\infty]$ satisfying for any~$a,a',b,b' \in X$
 \renewcommand{\theenumi}{\textit{(\alph{enumi})}}
 \begin{enumerate}
  \item\label{Item:DDextExtSym} $\DDext{a}{a'}{b}{b'} = \DDext{b}{b'}{a}{a'}$,
  \item\label{Item:DDextExtFlip} $\DDext{a}{a'}{b}{b'} = - \DDext{a'}{a}{b}{b'} = - \DDext{a}{a'}{b'}{b}$,
  \item\label{Item:DDextExtRepeat} $\DDext{a}{a}{b}{b'} = 0 = \DDext{a}{a'}{b}{b}$,
  \item\label{Item:DDextExtTransitive} $\DDext{a}{a'}{b}{b'} + \DDext{a'}{a''}{b}{b'} = \DDext{a}{a''}{b}{b'}$,
  \item\label{Item:DDextExtCocycle} $\DDext{a}{b}{c}{x} + \DDext{c}{a}{b}{x} + \DDext{b}{c}{a}{x} = 0$,
  \item\label{Item:CrossEqual} $\DDext{a}{a'}{b}{b'} = \infty$ if and only if $a=b' \in \partial X$ or~$a'=b \in \partial X$,
  \item\label{Item:StraightEqual} $\DDext{a}{a'}{b}{b'} = -\infty$ if and only if $a=b \in \partial X$ or~$a'=b' \in \partial X$,
  \item\label{Item:DDextEqvtDD} There are constants~$\alpha \geq 1$ and~$\beta \geq 0$ such that for all~$a,a',b,b' \in \VX$
  \[
   \frac{1}{\alpha} \DD{a}{a'}{b}{b'} - \beta
   \leq
   \DDext{a}{a'}{b}{b'}
   \leq
   \alpha \DD{a}{a'}{b}{b'} + \beta
  \]
  where $\DD{-}{-}{-}{-}$ denotes the double difference with respect to the metric~$\dG$ on~$\VX$.
 \end{enumerate}
\end{theorem}

\begin{proof}
 This is essentially~\cite[Theorem 35]{Mineyev2005FaJiMS}.
 First the double difference is extended to the set~$S \subseteq \overline{X}^4$ consisting of points~$(x,x',y,y')$ such that
 \begin{itemize}
  \item $x,y \in \partial X \Rightarrow x \neq y$,
  \item $x,y' \in \partial X \Rightarrow x \neq y'$,
  \item $x',y \in \partial X \Rightarrow x' \neq y$,
  \item $x',y' \in \partial X \Rightarrow x' \neq y'$.
 \end{itemize}
 Mineyev defines the extension of the double difference to~$S$ by using sequences of vertices converging to the appropriate boundary points.
 On~$S$ the double difference always takes values inside~$\RR$.
 Our metric~$\dext$ is an extension of~$\dhat$, so it coincides with~$\dhat$ on the vertex set (see Theorem~\ref{Thm:Extension}), so Mineyev's definition of the double difference on~$S$ carries over to our case.

We need to show it is well-defined, i.e. that it is independent of which sequence of points we use.
 Mineyev proved that it is independent of which sequence of vertices is used, but we need to know what happens for arbitrary sequences.
 When considering sequences of arbitrary points (i.e. not necessarily vertices) we know from Theorem~\ref{Thm:Extension} that~$\dext(x,y) = \Dhat(x,y)$ whenever the support of~$x$ is disjoint from the support of~$y$.
 So if~$(x_n)_{n \in \NN}$ is a sequence of points that converges to~$x \in \partial X$ and~$(y_n)_{n \in \NN}$ is a sequence of points that converges to~$y \in \overline{X} \backslash \lbrace x \rbrace$ then for sufficiently large~$n$, the support of~$x_n$ will be disjoint from the support of~$y_n$ and thus we can use the bilinear extension formula~\eqref{Eqn:Dhat} to say that the extension of the double difference is well-defined.
 
 To extend the double difference further from~$S$ to~$\overline{S}$, Mineyev shows that for any sequence of points in~$S$ converging to a point in~$\overline{S} \backslash S$ the double difference converges to~$\pm \infty$, with the sign depending on the point of~$\overline{S} \backslash S$, as specified in parts~\ref{Item:CrossEqual} and~\ref{Item:StraightEqual}.
This uses \cite[Proposition~33]{Mineyev2005FaJiMS} that says there exists a constant~$C' \geq 0$ such that for all~$u,v,w \in \VX$ if $w$ lies on a geodesic (with respect to~$\dG$) in~$\Graph$ then
\[
 \left\lvert
  \dhat(u,v) - \dhat(u,w) - \dhat(w,v)
 \right\rvert
 \leq C'.
\]
This holds for~$\dext$ too, using Corollary~\ref{Cor:QIExtension}\ref{Item:QIExtQI}.
Namely
\[
 \left\lvert
  \dext(u,v) - \dext(u,w) - \dext(w,v)
 \right\rvert
 \leq 3B' + C'.
\]
Hence Mineyev's argument is still applicable and gives parts~\ref{Item:CrossEqual} and~\ref{Item:StraightEqual}.
 
Finally, for part~\ref{Item:DDextEqvtDD}, Mineyev in~\cite[Theorem 35(h)]{Mineyev2005FaJiMS} shows there is an equivalence between~$\DDhat{-}{-}{-}{-}$ and~$\DD{-}{-}{-}{-}$ but we can go to~$\DDext{-}{-}{-}{-}$ using Remark~\ref{Rk:DoubleDifferences+Equivalent}.
\end{proof}

With this continuous extension of the double difference we can also extend the Gromov product as Mineyev does in \cite[Theorem~36]{Mineyev2005FaJiMS}.
Recall that the Gromov product of three points~$x,y,z$ in a metric space~$Z$ is
\[
 \GP{x}{y}{z} := \half \big( d(x,z) + d(y,z) - d(x,y) \big).
\]
It follows straight from the definitions that
\begin{equation}\label{Eqn:DDGP}
 \DD{a}{b}{x}{y} = \GP{b}{x}{a} - \GP{b}{y}{a}
\end{equation}
and it is this relation that we exploit to extend the Gromov product on the metric space~$(X,\dext)$ to~$\overline{X}$.
If we let~$\GPext{-}{-}{-}$ denote the Gromov product with respect to the metric~$\dext$ then the following theorem is an immediate corollary of Theorem~\ref{Thm:DDextExtensionToBoundary}, and this is our version of \cite[Theorem~36]{Mineyev2005FaJiMS}.
\begin{theorem}\label{Thm:GPextExtensionToBoundary}
 Let~$X$ be a hyperbolic complex.
 Set~$\overline{T}$ to be the subset of~$\overline{X}^3$ consisting of points~$(a,b,c)$ such that if~$c \in \partial X$ then~$a \neq c$ and~$b \neq c$. 
 Then the Gromov product~$\GPext{a}{b}{c}$ with respect to~$\dext$ extends to a continuous \IsomXinvt\ function~$\overline{T} \rightarrow [0,\infty]$ that satisfies
 \[
  \GPext{a}{b}{c} = \infty
  \text{ if and only if }
  c \in \partial X
  \text{ or }
  a = b \in \partial X.
 \]
\end{theorem}
\begin{proof}
 If~$(a,b,c) \in \overline{T}$ then~$(c,a,c,b) \in \overline{S}$ and we set~$\GPext{a}{b}{c} = \DDext{c}{a}{b}{c}$ using Theorem~\ref{Thm:DDextExtensionToBoundary}.
 This restricts to the original Gromov product (with respect to~$\dext$) on~$X^3$ via equation~\eqref{Eqn:DDGP} for the metric~$\dext$.
\end{proof}

Mineyev proves more properties of the double difference which we need to show still hold with our extension (instead of Mineyev's bilinear extension).
The following proposition looks at the convergence of the double difference, using inequality~\eqref{Eqn:DDhatConvergence}.

\begin{proposition}\label{Propn:Propn38}
 Let~$X$ be a hyperbolic complex.
 Let~$\overline{S}$ be as in Theorem~\ref{Thm:DDextExtensionToBoundary}.
 There exists a constant~$\lambda_0 \in [0,1)$ such that for all~$\lambda \in [\lambda_0,1)$ there is a~$T \geq 0$ such that for all~$(u,a,b,c) \in \overline{S}$ if $T \leq \max \lbrace \DDext{u}{a}{b}{c} , \DDext{u}{b}{a}{c} \rbrace =:m$ then
 \[
  \left\lvert
   \DDext{u}{c}{a}{b}
  \right\rvert
  \leq
  \lambda^m.
 \]
\end{proposition}
\begin{proof}
 This is essentially~\cite[Proposition~38]{Mineyev2005FaJiMS}.
 By the continuity of the extension of the double difference to~$\overline{S}$ (see Theorem~\ref{Thm:DDextExtensionToBoundary}) we only need to check the case~$u,a,b,c \in X$.
 If~$u,a,b,c$ are vertices then the proof here is the same as Mineyev's, since~$\dext = \dhat$ on vertices.

 In order to extend to arbitrary points in~$X$ we need to do a bit more work, since we need to show that under the assumptions given we can use the bilinear extension~$\Dhat$ of~$\dhat$, i.e. we want
 \[
  \DDext{u}{c}{a}{b} = \DDhat{u}{c}{a}{b}
 \] 
 so that we can write~$\DDext{u}{c}{a}{b}$ in terms of the double difference of vertices in the support of the points~$u,c,a,b$.

 Let~$A,B$ be the constants appearing in a quasi-isometry between~$\dhat$ and the word metric~$\dG$ on the vertex set of~$X$, so that
 \[
  \dext(x,y) = \min \left\lbrace \dhat(x,y) , 3(A+B)\dX(x,y) \right\rbrace
 \]
 which comes from setting~$C = A+B$ in  Definition~\ref{Defn:ExtensionMetric}.
 
 We want to find a condition when~$\dext(x,y) = \Dhat(x,y)$.
 We know from Corollary~\ref{Cor:QIExtension}\ref{Item:QIExtDisjointSupports} that~$\dext(x,y) = \Dhat(x,y)$ whenever the support of~$x$ is disjoint from the support of~$y$.
 If their supports are not disjoint then pick some vertex~$v$ in the support of both.
 It follows from Lemma~\ref{Lemma:l1PathMetricRestrictions}\ref{Item:dloneDiameterOfSimplex} that~$\dX(x,v) \leq 1$ and~$\dX(v,y) \leq 1$.
 Thus
 \[
  \dext(x,y)
  \leq \dext(x,v) + \dext(v,y)
  \leq 3C \, \dX(x,v) + 3C \, \dX(v,y)
  \leq 6C.
 \]
 Therefore, if~$\dext(x,y) > 6C$ then the support of~$x$ must be disjoint from the support of~$y$ and then~$\dext(x,y) = \Dhat(x,y)$.
 So we are interested in finding a lower bound for the distances used to define~$\DDext{u}{c}{a}{b}$.

 Without loss of generality suppose~$m = \DDext{u}{b}{a}{c}$.
 Equation~\eqref{Eqn:DDGP} for the metric~$\dext$ gives
 \begin{align*}
  \DDext{u}{b}{a}{c}
  &=
  \GPext{b}{a}{u} - \GPext{b}{c}{u}
  \\
  &\leq
  \min \left\lbrace \dext(u,a) , \dext(u,b) \right\rbrace
 \end{align*}
 and then the assumption~$T \leq m$ gives~$T \leq \dext(u,a),\dext(u,b)$.
 Moreover,
 \begin{align*}
  \DDext{u}{b}{a}{c} = - \DDext{c}{a}{u}{b}
  &= - ( \GPext{a}{u}{c} - \GPext{a}{b}{c} )
  \\
  &\leq
  \min \left\lbrace \dext(c,a) , \dext(c,b) \right\rbrace
 \end{align*}
 and so~$T \leq \dext(c,a),\dext(c,b)$ as well.
 
 Thus if we take~$T \geq 6(A+B)$ then for all four of these pairings~$\dext \equiv \Dhat$, and we get~$\DDext{u}{c}{a}{b} = \DDhat{u}{c}{a}{b}$.
 Since we are allowed to increase~$T$ as we wish, the rest of the proof works as in \cite[Proposition~38]{Mineyev2005FaJiMS}.
\end{proof}

Now Mineyev's results in \cite{Mineyev2005FaJiMS} hold for~$\dext$ instead of the bilinear extension~$\Dhat$ and the proofs are unchanged.
In particular...

%
%
%

\begin{description}
 \item ...the exponential~$e^{\DDext{x}{x'}{y}{y'}}$ defines a cross-ratio that is continuous on the set~$\overline{S} \subseteq \overline{X}^4$ as in \cite[Section~7]{Mineyev2005FaJiMS}.
This is analogous to the cross-ratio on the ideal boundary of a CAT$(-1)$-space (which was defined in \cite{Otal1992GeometrieSymplectiqueNegative}).
 \item ...the symmetric join~$\SymmJoinX$, and the metric~$\SymmJoinMetric$ on it, can be constructed as in \cite[Section~8]{Mineyev2005FaJiMS}.
 \item ...horofunctions (aka Busemann functions) and horospheres can be defined on $\SymmJoinX$ as in \cite[Section~10]{Mineyev2005FaJiMS}.
 \item ...the translation length can be defined for any isometry of $(X,\dext)$ as in \cite[Section~12]{Mineyev2005FaJiMS}.
\end{description}

\subsection{The Farrell-Jones Conjecture for hyperbolic groups}\label{Subsection:FJCHyperbolicGroups}

The proof of the $K$-theoretic Farrell-Jones Conjecture for an arbitrary hyperbolic group~$G$ was given in \cite{BartelsLueckReich2008KFJC} using special covers of~$G \times \overline{X}$, where~$\overline{X}$ is the (Gromov) compactification of the Rips complex of~$G$.
 (See \cite[Section~2]{BartelsLueckReich2008KFJC} for a summary of what goes into the proof.)

These covers were constructed in \cite{BartelsLueckReich2008ECHG} and uses Mineyev's symmetric join from \cite{Mineyev2005FaJiMS}.
Hence the bilinear extension formula~\eqref{Eqn:Dhat} is repeated by Bartels-L{\"u}ck-Reich in \cite[Subsection~6.1, p.~151]{BartelsLueckReich2008ECHG}.
But as explained at the start of Subsection~\ref{Subsection:Extension} this bilinear extension does not define a metric.

However, using the extension provided by Theorem~\ref{Thm:Extension} repairs this problem since the results of Mineyev in \cite{Mineyev2005FaJiMS} hold for this extension, as explained in Subsection~\ref{Subsection:MineyevsFlowSpace}.
This fixes this minor flaw in the proof of the $K$-theoretic Farrell-Jones Conjecture for hyperbolic groups.



\bibliography{../Bibliography/GroupTheory}
\bibliographystyle{alpha}

\end{document}